\documentclass[10pt]{amsart}
\usepackage{verbatim}
\usepackage{amsmath}
\usepackage{mathrsfs}
\usepackage[line,color,poly,all]{xy}
\usepackage{graphicx,url,xy,amssymb,booktabs}
\usepackage[noadjust]{cite}
\usepackage{xspace}

\newcommand{\RR}{\mathbb R}
\newcommand{\ZZ}{\mathbb Z}
\newcommand{\QQ}{\mathbb Q}
\newcommand{\CC}{\mathbb C}
\newcommand{\OO}{\mathcal O}

\DeclareMathOperator{\SL}{SL}
\DeclareMathOperator{\SO}{SO}
\DeclareMathOperator{\GL}{GL}

\DeclareMathOperator{\Tr}{Tr}
\DeclareMathOperator{\Norm}{N}

\renewcommand{\Re}{\operatorname{Re}}
\DeclareMathOperator{\diag}{diag}
\DeclareMathOperator{\Cone}{Cone}
\newcommand{\Vor}{Vorono\"{\i}\xspace}

\newcommand{\hp}{\mathfrak{H}}

\newcommand{\B}{\mathcal{B}}
\newcommand{\x}{\mathbf{x}}

\newcommand{\mat}[1]{\begin{bmatrix} #1 \end{bmatrix}}%

\theoremstyle{plain}
\newtheorem{thm}{Theorem}[section]
\newtheorem{lem}[thm]{Lemma}
\newtheorem{prop}[thm]{Proposition}

\theoremstyle{definition}
\newtheorem{defn}[thm]{Definition}

\renewcommand{\tilde}{\widetilde}
\newcommand{\bbb}{\mathfrak{b}}
\newcommand{\aaa}{\mathfrak{a}}
\newcommand{\fff}{\mathfrak{f}}
\newcommand{\innprod}[2]{\langle #1, #2\rangle}
\begin{document}

\title{Perfect forms over totally real number fields}
\author{Paul E.~Gunnells and Dan Yasaki}
\address{Department of Mathematics and Statistics\\University of Massachusetts\\Amherst, MA 01003-9305}
\address{Department of Mathematics and Statistics\\146 Petty Building\\ 
University of North Carolina at Greensboro\\Greensboro, NC 27402-6170}
\email{gunnells@math.umass.edu}
\email{d\_yasaki@uncg.edu}
\date{August 18, 2009}
\thanks{The first named author wishes to thank the
National Science Foundation for support of this research through NSF
grant DMS-0801214.}
\keywords{Perfect forms, modular symbols, well-rounded retract,
Eisenstein cocycle}
\subjclass{11H55, 11E10, 11F67, 11F75}
\begin{abstract}
A rational positive-definite quadratic form is \emph{perfect} if it
can be reconstructed from the knowledge of its minimal nonzero value
$m$ and the finite set of integral vectors $v$ such that $f (v) = m$.
This concept was introduced by \Vor and later generalized by Koecher
to arbitrary number fields.  One knows that up to a natural ``change
of variables'' equivalence, there are only finitely many perfect
forms, and given an initial perfect form one knows how to explicitly
compute all perfect forms up to equivalence.  In this paper we
investigate perfect forms over totally real number fields.  Our main
result explains how to find an initial perfect form for any such
field.  We also compute the inequivalent binary perfect forms over
real quadratic fields $\QQ (\sqrt{d})$ with $d\leq 66$.
\end{abstract}
\bibliographystyle{alpha}

\maketitle
\begin{section}{Introduction}

Let $f$ be a positive-definite rational quadratic form in $n$
variables.  Let $m (f)$ be the minimal nonzero value attained by $f$
on $\ZZ^{n}$, and let $M (f)$ be the set of vectors $v$ such that $f
(v) = m (f)$.  \Vor defined $f$ to be \emph{perfect} if $f$ is
reconstructible from the knowledge of $m (f)$ and $M (f)$
\cite{voronoi1}.  \Vor's theory was later extended by Koecher to a
much more general setting that includes quadratic forms over arbitrary
number fields $F$ \cite{koecher}.  Koecher also generalized a
fundamental result of \Vor , which says that modulo a natural $\GL_{n}
(\OO)$-equivalence, where $\OO$ is the ring of integers of $F$, there
are only finitely many $n$-ary perfect forms.  Moreover, there is an
explicit algorithm to determine the set of inequivalent perfect forms,
given the input of an initial perfect form \cite{voronoi1, Gmod}.

\Vor proved that the quadratic form $A_{n}$ is perfect for all $n$,
and using this was able to classify $n$-ary rational perfect forms for
$n\leq 5$.  In this paper, we consider totally real fields $F$ and
explain how to construct an initial perfect form.  Rather than trying
to give a closed form expression of such a form, we show how to use
the geometry of symmetric spaces and modular symbols to find an
initial perfect form.  A key role is played by the notion of
\emph{lattices of $E$-type} \cite{Kit}.  For $F$ real quadratic and
$n=2$, we carry out our construction explicitly to compute all
inequivalent binary perfect forms for $F=\QQ (\sqrt{d})$, $d\leq 66$.
These results complement work of Leibak and Ong \cite{leiants, lei6,
Ong}.

Our main interest in \Vor and Koecher's results is that they provide
topological models for computing the cohomology of subgroups
of $\GL_{n} (\OO)$, where $F$ is any number field.  This
cohomology gives a concrete realization of certain automorphic forms
that conjecturally have deep connections with arithmetic
geometry.  The cohomology of subgroups of $\GL_{n} (\ZZ)$, for
instance, has a (well known) relationship with holomorphic modular
forms when $n=2$, and for higher $n$ has connections with $K$-theory,
multiple zeta values, and Galois representations \cite{exp.ind, gs,
gonch}.  The cohomology of subgroups of $\GL_{2} (\OO)$, when $F$ is
totally real, is related to Hilbert modular forms \cite{freitag}.
However, computing these models for any example, a prerequisite for
using them to explicitly compute cohomology, is a nontrivial problem
as soon as $\OO \not = \ZZ$ or $n>2$.  Our work is a first step
towards more cohomology computations for $F$ totally real,
computations that we plan to pursue in the future.

\subsection*{Acknowledgements} We thank Avner Ash, Farshid Hajir, and
Mark McConnell for helpful conversations and their interest in our
work.

\end{section}

\begin{section}{Preliminaries}
Let $F$ be a totally real number field of degree $m$ with ring of
integers $\OO$.  Let $\iota = (\iota_1, \ldots , \iota_m)$ denote the
$m$ embeddings $F \to \RR$.  For $z \in F$, let $z^k$ denote
$\iota_k(z)$.  We extend this notation to other $F$-objects.  For
example, if $A=\mat{a_{ij}}$ is a matrix with entries in $F$,
then $A^k$ denotes the real matrix $A^k = \mat{a^k_{ij}}$.  An element
$z\in F$ is called totally positive if $z^{k}>0$ for each $k$.  We
write $z\gg 0$ if $z$ is totally positive.

\subsection{$n$-ary quadratic forms over $F$} An \emph{$n$-ary
quadratic form over $F$} is a map $f:\OO^n \to \QQ$ of the form
\begin{equation}\label{eq:evaluation}
f(x_1, \ldots, x_n) = \Tr_{F/\QQ}\biggl(\sum_{1\leq i,j\leq n} a_{ij} x_i x_j\biggr),
\quad \text{where $a_{ij} \in F$.}
\end{equation}
Our main object of study will be \emph{positive-definite} $n$-ary
quadratic forms over $F$. Specifically, since $\OO^n \simeq \ZZ^{nm}$,
$f$ can be viewed as a quadratic form on $\ZZ^{nm}$.  We require that
under this identification $f$ is a positive-definite quadratic form on
$\ZZ^{nm}$.  Equivalently, we can use the $m$ embeddings $\iota \colon
F \to \RR^{m}$ to view $f$ as a $m$-tuple $(f^1, \ldots , f^m)$ of
real quadratic forms of $\ZZ^n$.  It is easy to check that a quadratic
form over $F$ is positive-definite if and only if each $f^i$ is
positive-definite.

\subsection{Minimal vectors} There is a value $m(f)$ associated to
each positive-definite quadratic form $f$, called the \emph{minimum} of
$f$, given by
\[m(f) = \min_{v \in \OO^n\smallsetminus \{0 \}} f(v).\]
\begin{defn}
A vector $v \in \OO^n\smallsetminus \{0 \}$ is a \emph{minimal vector for
$f$} if $f(v)= m(f)$.  The set of minimal vectors is denoted $M (f)$.
\end{defn}
Note that $v$ is a minimal vector for $f$ if and only if $-v$ is as
well.  In our considerations the distinction between $v$ and $-v$ will
be irrelevant, and so we abuse notation and let $M (f)$ denote a set
of representatives for the minimal vectors modulo $\{\pm 1 \}$.

\subsection{Perfect forms} For most quadratic forms, knowledge of the
set $M (f)$ is not enough to reconstruct $f$.  A simple example is
provided by the one-parameter family of rational quadratic forms
$$f_{\lambda }(x,y)= x^{2}+\lambda xy+y^{2}, \quad \lambda \in (-1,1)\cap \QQ, $$
all of which are easily seen to satisfy $M (f_{\lambda}) = \{e_{1},
e_{2} \}$, where the $e_{i}$ are the standard basis vectors of
$\ZZ^{2}$.  On the other hand the rational binary form 
\[
g (x,y)= x^{2}+xy+y^{2}
\]
is reconstructible from the data of $\{M (g), m (g) \}$
(cf. \S\ref{ss:an}), which equals
\[
\bigl\{\{e_{1}, e_{2}, e_{1}-e_{2} \}, 1 \bigr\}.
\]
We formalize this notion, due to \Vor for rational quadratic
forms, with the following definition:
\begin{defn}[{\cite[\S 3.1]{koecher}}]
A positive-definite quadratic form $f$ over $F$ is said to be
\emph{perfect} if $f$ is uniquely determined by its minimum value
$m(f)$ and its minimal vectors $M(f)$.  That is, given the data $ \{M
(f), m (f) \} $, the system of linear equations
\begin{equation}\label{eq:linsystem}
\left\{\Tr_{F/\QQ}(v^tXv) = m(f)\right\}_{v \in M(f)}
\end{equation}
has a unique solution.
\end{defn}

We warn the reader that there are other notions of perfection for
quadratic forms over number fields in the literature, notably in the
work of Icaza \cite{icaza} and Coulangeon \cite{Cou}.  All notions
involve the reconstruction of $f$ from its minimal vectors, but these
authors use the norm where we have used the trace in the evaluation
\eqref{eq:evaluation} of a form on a vector in $\OO^{n}$.  Moreover,
Coulangeon uses a larger group to define equivalence of forms.
\end{section}

\begin{section}{Positive lattices}

\subsection{Lattices of $E$-type}
For these results we follow \cite{Kit}.
\begin{defn}
Let $V/\QQ$ be a vector space with positive-definite quadratic form
$\phi$.  A lattice $L \subset V$ is a \emph{positive lattice for
$\phi$} if for a $\ZZ$-basis $\B=\{e_1, \ldots, e_n\}$ of $L$, the
associated symmetric matrix for $\phi$ in $\B$-coordinates has
rational entries.
\end{defn}
We will denote a positive lattice for $\phi$ by pair $(L,\phi)$.  As
before, one can define minimal vectors and the minimum for a positive
lattice $(L,\phi)$.  We denote these $M(\phi)$ and $m(\phi)$, with $L$
understood.

Given two positive lattices $(L_1,\phi_1)$ and $(L_2,\phi_2)$, with
$L_1 \subset V_1$ and $L_2 \subset V_2$, one can construct a new
positive lattice $(L_1 \otimes L_2,\phi_1 \otimes \phi_2)$.
Specifically, let $B_1$ denote the symmetric bilinear form giving rise
to $\phi_1$, and let $B_2$ denote the symmetric bilinear form giving
rise to $\phi_2$.  We define a symmetric bilinear form $B$ on $V_1
\otimes V_2$ by first defining
\[B(v \otimes w, \hat v \otimes \hat w) = B_1(v , \hat v )B_2(w, \hat w)\]
on simple tensors and then by linearly extending to all of
$V_{1}\otimes V_{2}$.  Then one has a positive-definite quadratic form
$\phi = \phi_1 \otimes \phi_2$ on $V_1 \otimes V_2$ given by $\phi(x)
= B(x,x)$.  Note that by construction, we have $\phi(v \otimes w) =
\phi_1(v)\phi_2(w)$.

\begin{defn}
A positive lattice $(L,\phi)$ is of \emph{$E$-type} if
\[M(\phi \otimes \phi') \subset \{u \otimes v\;|\; u \in L, v \in L'\}\]
for every positive lattice $(L',\phi')$.
\end{defn}

In other words, a lattice is $E$-type if, whenever it is tensored with
another positive lattice, the minimal vectors of the tensor product
decompose as simple tensors.  Positive lattices of $E$-type are
particularly well-behaved with respect to tensor product:
\begin{prop}[{\cite[Lemma 7.1.1]{Kit}}]\label{prop:Etype}
Let $(L,\phi)$ and $(L', \phi')$ be positive lattices.  If $(L,\phi)$ is of $E$-type, then
\begin{enumerate}
\item $m(\phi \otimes \phi') = m(\phi)m(\phi'),$ and 
\item $M(\phi \otimes \phi') = \{x \otimes y \;|\; x \in M(\phi),\ y \in M(\phi')\}.$
\end{enumerate} 
\end{prop}
\subsection{The form $A_n$}\label{ss:an} We now give an example
that will be important in the sequel.  Let $A_{n}$ be the rational
quadratic form
\[A_n(x_1,\ldots , x_n) = \sum_{1 \leq i \leq j \leq n}x_ix_j.\]
It is easy to check, as was first done by \Vor \cite{voronoi1}, that
$A_n$ is perfect.  One computes \[m(A_n) = 1 \quad
\text{and} \quad M(A_n) = \{e_i\} \cup \{e_i - e_j\}.\]  According to 
\cite[Theorem~7.1.2]{Kit}, $(\ZZ^n, A_n)$ is of $E$-type.  

\subsection{Application to $n$-ary
quadratic forms over $F$} Fix $\alpha \in F$ totally positive.
Consider the $n$-ary quadratic form over $F$ given by
\begin{equation}\label{eqn:falpha}
f_\alpha(x_1, \ldots x_n) = \Tr_{F/\QQ}\left(\alpha A_n(x_1, \ldots , x_n)\right).\end{equation}

\begin{lem}\label{lem:split}
We have
\[(\OO^n,f_\alpha) = (\OO \otimes \ZZ^n, \phi_\alpha \otimes A_n),\]
where $\phi_\alpha(x) = \Tr_{F/\QQ}(\alpha x^2)$.
\end{lem}
\begin{proof}
It is clear that as $\ZZ$-modules, we have $\OO^n \simeq \OO \otimes \ZZ^n$.  Thus we
want to show that under this isomorphism, the quadratic form
$f_{\alpha}$ on $\OO^{n}$ is taken to the quadratic form $\phi_\alpha
\otimes A_n$ on $\OO \otimes \ZZ^n$.  We do this by explicit
computation.  Let $a \in \OO$
and let $\x = \sum x_ie_i \in \ZZ^n$.  Then we have
\begin{align*}
(\phi_\alpha \otimes A_n)(a \otimes \x) &= \phi_\alpha(a)A_n(\x)\\
&=\Tr_{F/\QQ}\left(\alpha a^2\right)\sum_{1 \leq i \leq j \leq n} x_i x_j\\
&=\Tr_{F/\QQ}\Bigl(\alpha \sum_{1 \leq i \leq j \leq n}a^2 x_ix_j\Bigr)\\
&=\Tr_{F/\QQ}\left(\alpha A_n(a\x)\right)\\
&=f_\alpha(a\x),
\end{align*}
which completes the proof.
\end{proof}

\begin{thm}\label{thm:scaledAnform}
Let $f_\alpha$ be as in \eqref{eqn:falpha}.  Then there exist nonzero
$\eta_1, \ldots , \eta_r\in \OO$ such that
\begin{enumerate}
\item the minimum of $f_\alpha$ is \[m(f_\alpha) = m(\phi_\alpha) =
\Tr_{F/\QQ}(\alpha \eta_i^2),\quad \text{and}\]
\item the minimal vectors of $f_\alpha$ are
\[M(f_\alpha) = \bigcup_{1 \leq k \leq r} \{\eta_k e_i\} \cup \{\eta_k(e_i-e_j)\}.\]
\end{enumerate}
\end{thm}
\begin{proof}
Since $(\ZZ^n,A_n)$ is of $E$-type, the result follows from
Proposition~\ref{prop:Etype} and Lemma~\ref{lem:split} by taking
$\{\eta_i\}$ to be the minimal vectors for $\phi_\alpha$.
\end{proof}

\end{section}

\begin{section}{The geodesic action, the well-rounded retract, and the
Eisenstein cocycle} 

In this section we present three tools that play a key role in the
proof of Theorem \ref{thm:main}.  The \emph{geodesic action} \cite{BS}
is an action of certain tori on locally symmetric spaces.  The
\emph{well-rounded retract} \cite{A2} is a deformation retract of
certain locally symmetric spaces.  The \emph{Eisenstein cocycle}
\cite{eisenstein, dedekind} is a cohomology class for $\SL_{m} (\ZZ)$ that gives
a cohomological interpretation of special values of the partial zeta
functions of totally real number fields of degree $m$.

\subsection{Geodesic action}\label{ss:ga}
Let $G$ be a semisimple connected Lie group, let $K$ be a maximal
compact subgroup, and let $X$ be the symmetric space $G/K$.  Fix a
basepoint $x \in X$.  This choice of basepoint determines a Cartan
involution $\theta_x$.  For a parabolic subgroup $P\subset G$, the
Levi quotient is $L_P = P/N_P$, where $N_P$ is the unipotent radical
of $P$.  Let $A_P$ denote the (real points of) the maximal $\QQ$-split
torus in the center of $L_P$, and let $A_{P,x}$ denote the unique lift
of $A_P$ to $P$ that is stable under the Cartan involution
$\theta_{x}$.

Since $P$ acts transitively on $X$, every point $z \in X$ can be
written as $z = p\cdot x$ for some $p \in P$.  Then Borel--Serre
define the geodesic action of $A_P$ on $X$ by
\[a \circ z  = (p \tilde{a}) \cdot x,\]
where $\tilde{a}$ is the lift of $a$ to $A_{P,x}$.
This action is independent of the choice of basepoint $x$, justifying
the notation.  Note that at the basepoint $x$, the geodesic action of
$A_{P}$ agrees with the ordinary action of its lift $A_{P,x}$.

\subsection{Well-rounded retract}\label{ss:wr} Now let $G=\SL_{m}
(\RR)$, $K=\SO (m)$.  The space $X$ is naturally isomorphic to the
space of $m$-ary positive-definite real quadratic forms modulo
homotheties.  Indeed, this follows easily from the Cholesky
decomposition from computational linear algebra: if $S$ is a symmetric
positive-definite matrix of determinant $1$, then there exists a
matrix $g\in G$ such that $gg^{t} = S$.

Let $W\subset X$ be the subset consisting of all forms whose
minimal vectors span $\RR^{m}$.  Then Ash proved that $W$ is an
$\SL_{m} (\ZZ)$-equivariant deformation retract of $X$.  Moreover $W$
naturally has the structure of a cell complex with polytopal cells,
and $\SL_{m} (\ZZ)$, and thus any finite-index subgroup $\Gamma\subset
\SL_{m} (\ZZ)$, act cellularly on $W$ with finitely many orbits.  The
retract can be used to compute the cohomology of $\Gamma$ for certain
$\ZZ \Gamma$-modules in the following way.  Let $M$ be a $\ZZ
\Gamma$-module attached to a rational representation of $\SL_{m}
(\QQ)$ and let $\tilde{M}$ be the associated local coefficient system
on $\Gamma \backslash X$.  Then we have isomorphisms
\[
H^{*} (\Gamma; M) \simeq H^{*} (\Gamma \backslash X; \tilde{M}) \simeq
H^{*} (\Gamma \backslash W; \tilde{M}).
\]

\subsection{Eisenstein cocycle}\label{ss:ec} As before let $\OO$ be
the ring of integers in a totally real number field of degree $m$.
Let $\bbb , \fff \subset \OO$ be relatively prime ideals.  The partial
zeta function $\zeta (\bbb , \fff ; s)$ attached to the ray class
$\bbb \pmod \fff$ is defined by the analytic continuation of the
Dirichlet series
\[
\zeta (\bbb , \fff ; s) = \sum_{\aaa} \Norm (\aaa)^{-s},\quad \Re
(s)>1,
\]
where the sum is taken over all integral ideals $\aaa$ such that $\aaa
\bbb^{-1}$ is principal with a totally positive generator in the coset
$1+\fff \bbb^{-1}$.  By the theorem of Klingen--Siegel \cite{siegel},
the special values $\zeta (\bbb, \fff ; 1-k)$, where $k\in \ZZ_{>0}$, are
rational. 

The special values have a cohomological interpretation.  Let $U$ be
the group of totally positive units in the coset $1+\fff$.  Sczech
constructed a sequence of rational cocycles $\eta (\bbb , \fff , k)\in
H^{m-1} (U; \QQ)$ which give the numbers $\zeta (\bbb, \fff ; 1-k)$ by
evaluation on the fundamental cycle in $H_{m-1} (U; \ZZ)$.  To
construct the cocycles $\eta (\bbb , \fff , k)$, one specializes a
``universal'' cocycle $\Psi \in H^{m-1} (\SL_{m} (\ZZ); M)$, where $M$
is a certain module.  After choosing $\bbb , \fff , k$, one plugs
$U$-invariant parameters into $M$ to obtain a module $M_{k}$, which is
$\CC$ with a nontrivial $\SL_{m} (\ZZ)$-action, and a class $\Psi(\bbb
, \fff , k)\in H^{m-1} (\SL_{m} (\ZZ); M_{k})$.  Then $\eta (\bbb ,
\fff , k)$ is obtained by restriction, after realizing $U$ as a
subgroup of $\SL_{m} (\ZZ)$ via a regular representation; note that
$M_{k}$ restricted to $U$ is trivial.

\end{section}

\begin{section}{Scaled trace forms and the main result}\label{sec:scaledtrace}

We now return to perfect forms.  The quadratic form $\phi_\alpha$ from
Lemma \ref{lem:split} will play an imporant role in Theorem
\ref{thm:main}, so we give it a name: 
\begin{defn}
For $\alpha \in F$, the \emph{scaled trace form associated to $\alpha$} is the map 
\[\phi_\alpha: \OO \to \QQ\] given by $\phi_\alpha(\eta) = \Tr_{F/\QQ}(\alpha \eta^2)$.
\end{defn}

For the remainder of the paper, fix a $\ZZ$-basis $\B =\{\omega_1,
\ldots, \omega_m\}$ of $\OO$.  Then for $\x = \sum x_i \omega_i \in
\OO$ with $x_i \in \ZZ$, we have
\[\phi_\alpha(\x) = \sum_{1 \leq i,j \leq m}\Tr_{F/\QQ}(\alpha \omega_i \omega_j)x_i x_j.\]
In particular, fixing $\B$ allows us to view the form $\phi_\alpha$ as
an $m$-ary quadratic form $[\phi_\alpha]_\B$ over $\QQ$.

Let $V$ denote the $m(m+1)/2$-dimensional $\RR$-vector space of
$m$-ary quadratic forms over $\RR$.  Note that for $\alpha , \beta \in
F$, we have
\[[\phi_{\alpha+\beta}]_\B = [\phi_\alpha]_\B + [\phi_\beta]_\B.\] In
particular the image of $F\otimes \RR$ in $V$ is an $m$-dimensional
subspace.  The form $\phi_\alpha$ is positive-definite if $\alpha \gg
0$.  Let $C \subseteq V$ denote the real cone of positive-definite $m$-ary
quadratic forms, and let $C^+_\B$ denote the subcone corresponding to
the totally positive scaled trace forms.  More precisely, let
\[C^+_\B=\Cone(\{[\phi_\alpha]_\B \;|\; \alpha \gg 0\}) \otimes \RR.\]

Let $X$ be the global symmetric space $G/K$, where $G = \SL_{m} (\RR)$
and $K=\SO (m)$.  Recall (\S\ref{ss:wr}) that we have an isomorphism $X\simeq
C/\RR_{>0}$, where $\RR_{>0}$ acts on $C$ by homotheties.  We denote
by $X^{+}_{\B}$ the image of $C^+_{\B}$ in $X$ under the projection
$C\rightarrow X$.

\begin{lem}\label{lem:geodesic}
There exists an $\RR$-split torus $A \subset G$ and a point
$x_\B\in X$ such that 
\[
X^+_\B = \{a \circ x_\B\;|\;a \in A\},
\]
where $\circ$ is the geodesic action of $A$ on $X$ (\S\ref{ss:ga}).
\end{lem}
\begin{proof}
We begin with some computations in $C$.  For $\alpha \in F$
totally positive, let $S (\alpha ) = \mat{S (\alpha )_{ij}}$ denote
the positive-definite symmetric $m \times m$ matrix corresponding to
$[\phi_\alpha]_\B$.  Thus
\[S (\alpha )_{ij} = \Tr_{F/\QQ}(\alpha
\omega_i\omega_j) = \sum_{1\leq k \leq m}
\alpha^k\omega_i^k\omega_j^k.
\] 
We can write $S (\alpha )$ as $g (\alpha ) g (\alpha )^t$, where $g
(\alpha )$ is given by
\[g (\alpha )_{ij} = \sqrt{\alpha^j}\omega_i^j.\]
This implies $g (\alpha) =\Omega a$, where
\begin{equation}\label{eq:omegaanda}
\Omega_{ij} = \omega_i^j\quad \text{and} \quad a=\diag(\sqrt{\alpha^1}, \ldots, \sqrt{\alpha^m}).
\end{equation}
From these considerations it is clear that the cone $C^{+}_{\B}\subset
C$ is given by the set of matrices of the form $\Omega a$, where $a$
is allowed to vary over \emph{all} positive real diagonal matrices,
not just those of the special form on the right of
\eqref{eq:omegaanda}.

Now we pass to $X$ by modding out by homotheties.  Let $x_{\B}$ be the
image of the point $[\phi_{1}]_{\B}$.  Let $\Upsilon$ be the unique
positive multiple of $\Omega$ such that $\det (\Upsilon) = 1$; this
also maps onto $x_{\B}$.  Then the subset of $C$ given by
\[
\{\Upsilon  a \mid a = \diag (a_{1}, \dotsc , a_{m}), a_{k}\in \RR_{>0},
a_{1}\dotsb a_{m}=1\}
\]
maps diffeomorphically onto $X^{+}_{\B}\subset X$.  Let
$P_{\infty}\subset G$ be the parabolic subgroup of upper-triangular
matrices, and let $A_{\infty} \subset P_{\infty}$ be the diagonal
subgroup.  Let $x$ denote the point of $X$ fixed by $K$. Note that
$X_{\B}^+$ is precisely a translate of the submanifold defined by the
geodesic action of $A_\infty$ on $x$:
\[
X_{\B}^+ = \Upsilon \cdot \{a \circ x \mid a \in A_\infty\}.
\]
By ``transport de structure'' we can express this at the basepoint,
that is
$x_\B = \Upsilon \cdot x$
\[X_\B = \{b \circ x_\B \mid b \in \Upsilon A_\infty \Upsilon^{-1}\}.\]
More precisely, 
\begin{equation}\label{eq:geodesic} 
(\Upsilon a) \cdot x = (\Upsilon a \Upsilon^{-1}) \Upsilon \cdot x = (\Upsilon a \Upsilon^{-1}) \cdot x_\B.
\end{equation} 

Now $\Upsilon a \Upsilon^{-1} \in \Upsilon A_\infty \Upsilon^{-1}$,
and \eqref{eq:geodesic} is exactly the geodesic action of the element
$\Upsilon a \Upsilon^{-1}$ on the point $x_\B $.  Thus we may take $A
= \Upsilon A_{\infty} \Upsilon^{-1}$, and the lemma follows.
\end{proof}

By Theorem~\ref{thm:scaledAnform}, to find a perfect $n$-ary form over
$F$, one can look for scaled trace forms with many linearly
independent minimal vectors.  Specifically, a scaled trace form
$f_\alpha = \phi_\alpha \otimes A_n$ is perfect if
\begin{equation}\label{eq:linsystemscaled}
\left\{\Tr_{F/\QQ}(v^tXv) = m(f_\alpha)\right\}_{v \in M(f_\alpha)}
\end{equation} defines $mn(n + 1)/2$ linearly independent conditions
on the space of quadratic forms.  Since $A_n$ is perfect over $\QQ$,
the minimal vectors of $A_n$ define $n(n+1)/2$ linearly independent
conditions.  Thus by Theorem \ref{thm:scaledAnform}, if $\phi_\alpha$
has $m$ linearly independent minimal vectors, then
\eqref{eq:linsystemscaled} will impose $mn(n+1)/2$ linearly
independent relations in the space of quadratic forms over $F$, and
$f_{\alpha}$ will be perfect. We now prove our main result, which
asserts that such an $\alpha$ can always be found:

\begin{thm}\label{thm:main}
There exists $\alpha \in F$ totally positive such that $\phi_{\alpha}$
has $m$ linearly independent minimal vectors, and thus
\[f_{\alpha }(x_1, \ldots, x_n) = \Tr_{F/\QQ}\left(\alpha A_n(x_1, \ldots ,  x_n)\right)\] 
is a perfect form over $F$.
\end{thm}
\begin{proof}
We must show that we can find $\alpha \gg 0$ such that
$[\phi_\alpha]_\B$ is well-rounded.  By Lemma~\ref{lem:geodesic}, we
know that a choice of basis $\B$ gives rise to a point $x_\B$ and a
maximal $\RR$-split torus $A$ such that $X^{+}_{\B} = \{a \circ
x_{\B}\mid a\in A\}$.  We will show that $X^+_\B \cap W \not =
\emptyset $, where $W$ is the retract for $\Gamma =\SL_{m} ( \ZZ)$
(\S\ref{ss:wr}), and that the inverse image in $C$ of any point in
$X^+_\B \cap W$ is a ray containing an $F$-rational point
$\phi_\alpha$.  This will prove the theorem.

To show $X^+_\B \cap W \not = \emptyset $ we use the Eisenstein
cocycle (\S\ref{ss:ec}).  Let $\fff = \OO$, so that $\zeta (\bbb ,
\fff ; s)$ is the Dedekind zeta function $\zeta_{F} (s)$, and $U$ is
the group of totally positive units.  Abbreviate $\Psi(\bbb , \fff ,
k)$ (respectively, $\eta(\bbb , \fff ,k)$) to $\Psi (k)$ (resp.,
$\eta (k)$).  Using the regular representation attached to the basis
$\B$, we have an injection $i\colon U\rightarrow \Gamma $.  Let
$M_{k}'$ be the module dual to $M_{k}$.  Since $M_{k}$ and $M'_{k}$
are trivial after restriction to $U$, we obtain induced maps
$i_{*}\colon H_{m-1} (U; \CC) \rightarrow H_{m-1} (\Gamma ; M'_{k})$
and $i^{*}\colon H^{m-1}(\Gamma ; M_{k}) \rightarrow H^{m-1} (U;
\CC)$.

Let $\innprod{\phantom{a}}{\phantom{a}}_{\star}$ the pairing between
$H^{m-1}$ and $H_{m-1}$, where $\star$ is either $U$ or $\Gamma $, and
where the target is $\CC \simeq M_{k}\otimes_{\Gamma} M'_{k}$.  Let
$\xi \in H_{m-1} (U;
\CC)$ be the fundamental class.  Then 
\[
\innprod{i_{*} (\xi)}{\Psi(k)}_{\Gamma} = \innprod{\xi}{i^{*}
(\Psi(k))}_{U} = \innprod{\xi}{\eta (k)}_{U} = \zeta_{F}(1-k).
\]
Since the special values $\zeta_{F} (1-k)$ do not vanish identically,
the class $i_{*} (\xi)$ pairs nontrivially with $\Psi (k)$ for some
$k$.  But it is easy to see that $i_{*} (\xi)$ is the same as the
class of $X^{+}_{\B} \pmod \Gamma $ in the homology of the quotient
$\Gamma \backslash X$.  If $X^{+}_{\B} \cap W$ were empty, then by the
discussion in \S\ref{ss:wr} the class of $X^{+}_{\B} \pmod \Gamma $
would pair trivially with all cohomology classes for all coefficient
modules, which is a contradiction.  Thus $X^{+}_{\B} \cap W\not
=\emptyset$.

To finish we must prove that the ray above an intersection point in
$X^{+}_{\B} \cap W$ contains a form $\phi_{\alpha}$ with $\alpha \in
F$.  This follows easily since $W$ is cut out by linear equations with
$\QQ$-coefficients and from the explicit form of the cone
$C^{+}_{\B}$.  This completes the proof of the theorem.
\end{proof}

\end{section}
\begin{section}{Real quadratic fields}
\subsection{Preliminaries} Let $d$ be a square-free positive integer,
and $\OO$ be the ring of integers in the real quadratic field
$F=\QQ(\sqrt{d})$.  Then $\OO$ is a $\ZZ$-lattice in $\RR^2$,
generated by $1$ and $\omega$, where $\omega=(1+\sqrt{d})/2$ if $d
\equiv 1 \pmod 4$ and $\omega=\sqrt{d}$ otherwise.  The discriminant
$D$ equals $d$ if $d \equiv 1 \pmod 4$ and equals $4d$ otherwise.

\subsection{Scaled trace forms} Let $C$ be the cone of
positive-definite binary quadratic forms.  Modding out by homotheties,
we can identify $C/\RR_{>0}$ with the upper half-plane $\hp$.  One
such identification is given by 
\begin{equation}\label{eq:identification}
x+iy \longmapsto \mat{1&-x\\-x&x^2 +
y^2}.
\end{equation}
Fixing a $\ZZ$-basis $\B = \{1,\omega\}$ for $\OO$, we
consider the subcone $C^+_\B \subset C$ of totally positive scaled
trace forms as described in \S\ref{sec:scaledtrace}.

By Lemma~\ref{lem:geodesic}, it follows that the image of
$C^+_\B$ is a geodesic in $\hp$.  Considering
\eqref{eq:identification} and our choice of basis,
we see that $C^+_\B$ corresponds to
\[
\left\{\left.\mat{\Tr_{F/\QQ}(\alpha) & \Tr_{F/\QQ}(\alpha \omega) \\
\Tr_{F/\QQ}(\alpha \omega) & \Tr_{F/\QQ}(\alpha \omega^2) }\ \right|\
\alpha \gg 0\right \} \otimes \RR.
\] 
On $\hp$ this becomes the geodesic $X^{+}_{\B}$ defined by
\begin{align*}
(x + \frac{1}{2})^2 + y^2 &= \frac{d}{4} \quad \text{if $d \equiv 1 \pmod 4$,}\\
x^2 + y^2 &= d \quad \text{otherwise.}
\end{align*}
The well-rounded retract $W\subset \hp$ is the infinite trivalent tree
shown in Figure~\ref{fig:tree}.  The crenellation comes from arcs of
circles of the form $(x-n)^2 + y^2 = 1$, where $n \in \ZZ$.  One can
compute a point $x_0+iy_0$ of the intersection of $W$ and the geodesic
corresponding to $C^+_\B$.  Let
\[
X(n) = \begin{cases}
\displaystyle{\frac{4n^2+d-5}{4+8n}} &\text{if $d \equiv 1 \pmod 4$,}\\
\displaystyle{\frac{n^2+d-1}{2n}} &\text{otherwise.}
\end{cases}
\]
Then $x_0 = \min_{n \in \ZZ}|X(n)|$, and $y_0$ can be explicitly
computed from $x_0$.  Specifically, let $\tilde{n}$ be a non-negative
integer such that $X(\tilde{n}) = x_0$.  Then $y_0$ satisfies
\[(x_0-\tilde{n})^2+y_0^2 = 1.\] The corresponding scaled trace form
is $\phi_{\alpha}$, where
\[\alpha = \begin{cases}
\displaystyle{ \frac{d-(2x_0+1)\sqrt{d}}{2d}}&\text{if $d \equiv 1 \pmod 4$,}\\
 \displaystyle{\frac{d-x_0\sqrt{d}}{2d}}&\text{otherwise.}
\end{cases}
\] 

\begin{figure}[htb]\label{fig:tree}
\includegraphics[width=0.7\textwidth]{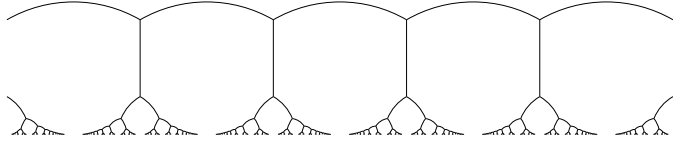}
\caption{The well-rounded retract $W \subset \hp$.}
\end{figure}

We summarize the results of this computation below.
\begin{prop}\label{prop:realquadtraceform}
Let $\alpha$ and $\tilde{n}$ be as above.  Let $\eta =
\tilde{n}+\omega$.  Then the scaled trace form $\phi_\alpha$ is
minimized at $\{\pm 1, \pm \eta\}$.
\end{prop}
\subsection{Binary perfect quadratic forms} Given a binary perfect
form as an initial input, there is an algorithm to compute the
$\GL_2(\OO)$-equivalence classes of binary perfect forms over $F$
\cite{Gmod}.  This was done in \cite{Ong} for $F = \QQ(\sqrt{2})$,
$\QQ(\sqrt{3})$, $\QQ(\sqrt{5})$ and in \cite{lei6} for
$F=\QQ(\sqrt{6})$. Using Proposition~\ref{prop:realquadtraceform}, we
compute the $\GL_2(\OO)$-equivalence classes of these forms over
$F=\QQ(\sqrt{d})$ for square-free $d\leq 66$.  The computations were
done using \verb+Magma+ \cite{magma} and \verb+PORTA+ \cite{porta}.
The number of $\GL_2(\OO)$-classes of perfect forms is given in
Table~\ref{tab:perfectforms}.  Figure \ref{graph:perfect} shows a plot
of the data $(D,N_{D})$ from Table \ref{tab:perfectforms}.

\begin{table}[htb]
\[
\begin{array}{|ccc|ccc|ccc|ccc|}\hline
D&h_{D}&N_{D}&D&h_{D}&N_{D}&D&h_{D}&N_{D}&D&h_{D}&N_{D}\\\hline
5 & 1 & 2 & 37 & 1 & 97 & 76 & 1 & 404 & 168 & 2 & 918 \\
8 & 1 & 2 & 40 & 2 & 52 & 88 & 1 & 645 & 172 & 1 & 2683 \\
12 & 1 & 3 & 41 & 1 & 144 & 92 & 1 & 343 & 184 & 1 & 4306 \\
13 & 1 & 9 & 44 & 1 & 66 & 104 & 2 & 326 & 188 & 1 & 1358 \\
17 & 1 & 34 & 53 & 1 & 93 & 120 & 2 & 524 & 204 & 2 & 1992 \\
21 & 1 & 16 & 56 & 1 & 144 & 124 & 1 & 1473 & 220 & 2 & 2397 \\
24 & 1 & 22 & 57 & 1 & 515 & 136 & 2 & 1135 & 232 & 2 & 3042 \\
28 & 1 & 30 & 60 & 2 & 89 & 140 & 2 & 487 & 236 & 1 & 2977 \\
29 & 1 & 39 & 61 & 1 & 206 & 152 & 1 & 995 & 248 & 1 & 2360 \\
33 & 1 & 149 & 65 & 2 & 301 & 156 & 2 & 941 & 264 & 2 & 3139 \\
\hline
\end{array}\]
\caption{$\GL_2(\OO)$-classes of perfect binary quadratic forms over
real quadratic fields.  The discriminant is $D$, the class number of
$\QQ (\sqrt{D})$ is $h_{D}$, and the number of inequivalent forms is $N_{D}$.}
\label{tab:perfectforms}
\end{table}

\begin{figure}[htb]
\begin{center}
\input{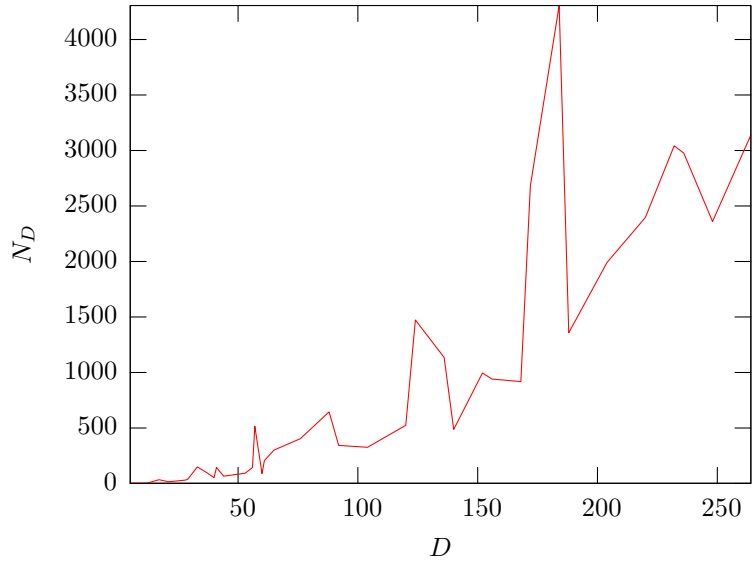}
\end{center}
\caption{Perfect forms by discriminant\label{graph:perfect}}
\end{figure}

\end{section}
\bibliography{perfectforms}   
\end{document}